\documentclass[11pt, a4paper,reqno]{amsart}

\usepackage{color} \definecolor{bleu_sombre}{rgb}{0,0,0.6}  \definecolor{rouge_sombre}{rgb}{0.8,0,0}\definecolor{vert_sombre}{rgb}{0,0.6,0}
\usepackage[plainpages=false,colorlinks,linkcolor=bleu_sombre,citecolor=rouge_sombre,urlcolor=vert_sombre,breaklinks]{hyperref}

\usepackage[english]{babel}

\usepackage{amsmath,amssymb,amsthm,graphicx,amsfonts,url,color,enumerate,dsfont,stmaryrd,mathabx}

\theoremstyle{plain}
\newtheorem{theorem}{{Theorem}}[section] 
\newtheorem*{theorem*}{{Theorem}}
\newtheorem{proposition}[theorem]{Proposition}
\newtheorem*{proposition*}{Proposition}

\newtheorem*{corollary*}{Corollary}
\newtheorem{lemma}[theorem]{Lemma}
\newtheorem*{lemma*}{Lemma}

\theoremstyle{definition}

\newtheorem*{definition*}{Definition}

\theoremstyle{remark}

\makeatletter

\@addtoreset{equation}{section}  
\makeatother

\newcommand{\abs}[1]{\left\vert #1\right\vert}      
\newcommand{\nr}[1]{\left\Vert #1\right\Vert}          
\newcommand{\innp}[2]{\left \langle #1 , #2 \right \rangle}           

\newcommand{\set}[1]{\left\{ #1 \right\}}		

\renewcommand{\leq}{\leqslant}	\renewcommand{\geq}{\geqslant}
\newcommand{\inv}{^{-1}}

\newcommand{\1}{\mathds 1}       
\renewcommand{\Re}{\mathsf{Re}}        
\renewcommand{\Im}{\mathsf{Im}}
\newcommand {\limt}[2]{\xrightarrow[#1 \to #2]{}}

\newcommand{\Dom}{\mathsf{Dom}}
\newcommand{\Sp}{\s}

\newcommand{\loc}{{\mathsf{loc}}}

\newcommand{\R}{\mathbb{R}}		\newcommand{\C}{\mathbb{C}}
\newcommand{\N}{\mathbb{N}}

\newcommand{\Lc}{{\mathcal L}}

\renewcommand{\a}{\alpha}\renewcommand{\b}{\beta}\newcommand{\G}{\Gamma}\renewcommand{\d}{\delta}\newcommand{\e}{\varepsilon}\renewcommand{\th}{\theta}\renewcommand{\k}{\kappa}\newcommand{\s}{\sigma}\newcommand{\f}{\varphi}\renewcommand{\O}{\Omega}

\newcommand{\HH}{\mathscr H}
\newcommand{\Wc}{\mathcal W}
\newcommand{\Rc}{\mathcal R}

\newcommand{\Nc}{\mathcal N}

\newcommand{\stepp}{\noindent {\bf $\bullet$}\quad }

\usepackage{mathrsfs}

\newcommand{\detail}[1]
{
}

\newenvironment{psmallmatrix}
  {\left(\begin{smallmatrix}}
  {\end{smallmatrix}\right)}
\begingroup
    \makeatletter
    \@for\theoremstyle:=definition,remark,plain\do{%
        \expandafter\g@addto@macro\csname th@\theoremstyle\endcsname{%
            \addtolength\thm@preskip\parskip
            }%
        }
\endgroup
\usepackage{color}
\usepackage[normalem]{ulem}
\definecolor{DarkBlue}{rgb}{0,0.1,0.7}  
\definecolor{DarkGreen}{rgb}{0,0.5,0.1}

\newcommand\soutD{\bgroup\markoverwith
{\textcolor{DarkGreen}{\rule[.5ex]{2pt}{1pt}}}\ULon}
\newcommand\soutJ{\bgroup\markoverwith
{\textcolor{DarkBlue}{\rule[.5ex]{2pt}{1pt}}}\ULon}
\newcommand{\Hm}[1]{\leavevmode{\marginpar{\tiny%
$\hbox to 0mm{\hspace*{-0.5mm}$\leftarrow$\hss}%
\vcenter{\vrule depth 0.1mm height 0.1mm width \the\marginparwidth}%
\hbox to
0mm{\hss$\rightarrow$\hspace*{-0.5mm}}$\\\relax\raggedright #1}}}

\title[Wave equation with Dirac damping]{Spectrum of the wave equation with Dirac damping on a non-compact star graph}
\author{David Krej\v{c}i\v{r}\'{i}k}
\address[D.~Krej\v{c}i\v{r}\'{i}k]{Department of Mathematics, Faculty of Nuclear Sciences and Physical Engineering, Czech Technical University in Prague, Trojanova 13, 120 00, Prague, Czech Republic}
\email{david.krejcirik@fjfi.cvut.cz}
\author{Julien Royer}
\address[J.~Royer]{Institut de math\'ematiques de Toulouse, Universit\'e Toulouse 3, 
118 route de Narbonne, F-31062 Toulouse cedex 9, France}
\email{julien.royer@math.univ-toulouse.fr}
\date{26 April 2022}

\begin{document}

\begin{abstract}
We consider the wave equation on non-compact star graphs,
subject to a distributional damping 
defined through a Robin-type vertex condition 
with complex coupling.
It is shown that the non-self-adjoint generator of the evolution problem
admits an abrupt change in its spectral properties 
for a special coupling related to the number of graph edges.
As an application, we show that the evolution problem
is highly unstable for the critical couplings.
The relationship with the Dirac equation in non-relativistic 
quantum mechanics is also mentioned.
\end{abstract}

\maketitle
 
\section{Introduction}
%
The primary motivation of this paper is 
the one-dimensional wave equation
\begin{equation}\label{wave}
  \psi_{tt} - \alpha \delta \psi_t - \psi_{xx} = 0 \,,
\end{equation}
where $t>0$ and $x \in \R$
are the time and space variables, respectively, 
$\delta$ is the Dirac delta function 
and~$\alpha$ is a complex number.  
The distribution $- \alpha \delta$ models a highly localised damping:
dissipation or supply of energy if~$\alpha$ is negative or positive, 
respectively, while purely imaginary~$\alpha$ admits 
a conservative, quantum-mechanical interpretation.  

In the case of 
$\alpha$ real and the space variable~$x$ restricted to a \emph{bounded} interval,
say $(-\frac{\pi}{2},\frac{\pi}{2})$,
the model~\eqref{wave} was introduced
in~\cite{Bamberger-Rauch-Taylor_1982} 
in order to explain the playings of harmonics on stringed instruments.
By a detailed spectral analysis, 
the authors argue that the optimal damping is $\alpha=-2$.
The problem has been more recently analysed in
\cite{Ammari-Henrot-Tucsnak_2000,Ammari-Henrot-Tucsnak_2001,Cox-Henrot_2008}
(see also \cite[Sec.~4.1.1]{Ammari-Nicaise}).
Qualitatively, spectral properties of~\eqref{wave} in the bounded case
are the expected ones: 
the spectrum is composed of isolated eigenvalues
of finite algebraic multiplicities.
However, there is an abrupt change in basis properties 
depending on whether $|\alpha| \not = 2$ or $|\alpha| = 2$.
While the root vectors form a Riesz basis in the former case,
they are not even complete for the special values $\alpha=\pm 2$.

The objective of the present paper is to point out 
that the peculiar spectral transition is even more drastic in the \emph{unbounded} situation considered here. Moreover, we allow for~$\alpha$ being an arbitrary complex number). If $\alpha \in \C \setminus \{\pm 2\}$,
the spectrum coincides with the imaginary axis~$i\R$, 
as in the damped-free case $\alpha=0$.
However, the spectrum abruptly fills in the whole complex half-plane
$
  \set{z \in \C \, : \, \pm \Re(z) \geq 0}
$  
as long as $\alpha = \pm 2$. 
This phenomenon was previously announced in \cite[Rem.~1]{KK8},
but no rigorous analysis was provided.
Wild spectral properties for the damped wave equation
with different unbounded dampings have been recently observed in
\cite{Freitas-Siegl-Tretter_2018,Freitas-Hefti-Siegl_2020}.

What is more, we give an insight into the appearance 
of the ``magical'' values~$\pm 2$ by considering~\eqref{wave}
in the more general situation of non-compact star graphs
with $N \geq 1$ edges. The real line~$\R$ can be considered as the star graph with two edges, 
while the case $N= 1$ corresponds to the wave equation 
on the half-line with damping at the boundary.  
It turns out that the abrupt change in spectral properties 
happens precisely for $\alpha = \pm N$.

The Laplacian on metric graphs with non-self-adjoint
coupling conditions at the vertices has been recently analysed in
\cite{HKS,Royer-Riviere_2020,Hussein_2021}. 
The damped wave equation on metric graphs 
is considered in \cite{AbdallahMerNic13,Freitas-Lipovsky_2017},
however, just self-adjoint coupling conditions at the vertices 
(and finite edges) are considered. See also~\cite{AsselJeKh16} for 
an unbounded network with Dirichlet and Kirchhoff conditions.
The present paper therefore opens a new direction 
of mathematically interesting and physically relevant research.
 
The organisation of this paper is as follows. 
In Section~\ref{Sec.results} we introduce our general model
and formulate the main results.
The preliminary Section~\ref{Sec.general} collects 
basic properties of the wave operator. 
In Section~\ref{Sec.limit} we show that the Dirac damping
can be realised as a limit of properly scaled regular dampings,
in a norm-resolvent sense.
The proofs of more involved spectral and evolution results 
are given in Sections~\ref{Sec.proofs} and~\ref{Sec.evolution}, respectively.
In the concluding Section~\ref{Sec.Dirac},
we provide a relationship between the damped wave equation
and the Dirac equation in relativistic quantum mechanics,
in order to motivate the present setting of unbounded geometries
and non-real dampings.

\section{Our model and main results}\label{Sec.results}
%

\subsection{The damped wave equation and the wave operator}
Let $N \in \N^* =\{1,2,\dots\}$. 
We set $\Nc = \{1,\dots,N\}$. We consider a metric graph $\G$ given by $N$ edges of infinite length linked together at a central vertex. Concretely, $\G$ consists of $N$ copies of the open half-line $\R_+^* = ]0,+\infty[$ and the link between the edges will be encoded in the domain of our operator (see \eqref{Dom-Wc} below).

We set 
\[
L^2(\G) = L^2(\R_+^*)^N.
\]
It is endowed with its natural Hilbert structure. For $k \in \N^*$ we similarly define
\[
\dot H^k (\G^*) =  \dot H^k(\R_+^*) ^N , \quad H^k (\G^*) =  H^k(\R_+^*) ^N,
\]
and for $u = (u_j)_{j \in \Nc} \in \dot H^k(\G^*)$ we write $u^{(k)}$ for $(u_j^{(k)})_{j \in \Nc}$. In the sequel, when we consider a function $u$ on $\G \simeq (\R_+^*)^N$ it will be implicitly understood that we denote by $(u_j)_{j \in \Nc}$ its components.

We say that $u \in \dot H^1(\G^*)$ is \emph{continuous} at~$0$ if 
\[
\forall j,k \in \Nc, \quad u_j(0) = u_k(0).
\]
In this case we denote by $u(0)$ the common value of $u_j(0)$, $j \in \Nc$. We set 
\[
\dot H^1(\G) = \set{ u \in \dot H^1(\G^*) \, : \, \text{$u$ is continuous at 0}}.
\]
To define the natural Hilbert structure, we have to take the quotient of $\dot H^1(\G)$ by constant functions. We denote by $\tilde H^1(\G)$ this quotient. Then $\tilde H^1(\G)$ is a Hilbert space for the inner product defined by 
\[
\forall u,v \in \tilde H^1(\G), \quad \innp{u}{v}_{\tilde H^1(\G)} = \innp{u'}{v'}_{L^2(\G)} =  \sum_{j= 1}^N \innp{u_j'}{v_j'}_{L^2(\R_+^*)}.
\]
Throughout the paper we will identify a function in $\dot H^1(\G)$ and its equivalence class in $\tilde H^1(\G)$. This implicitly means that the results should be invariant by addition of a constant to the functions in $\dot H^1(\G)$.

We consider the damped wave equation on $\G$ with damping at the central vertex. 
Let~$I$ be an interval of $\R$ which contains 0. We introduce the problem  
\begin{equation} \label{eq:wave}
\partial_{tt} u_j(t,x) - \partial_{xx} u_j(t,x) = 0, \quad  \forall j \in \Nc, \forall t \in I, \forall x \geq  0
\end{equation}
with the continuity at the vertex
\begin{equation} \label{eq:wave-continuite}
u_j(t,0) = u_k(t,0),  \quad \forall j,k \in \Nc, \forall t \in I
\end{equation}
(this also implies the continuity at the vertex for $\partial_t u$) 
and the damping condition
\begin{equation} \label{eq:wave-Robin}
\sum_{j=1}^N \partial_x u_j(t,0) + \a \partial_t u(t,0) = 0, \quad  \forall t \in I,
\end{equation}
where $\a \in \C$. The Cauchy problem is completed by the initial conditions
\begin{equation} \label{eq:wave-CI}
(u_j,\partial_t u_j)|_{t = 0} = (f_j,g_j), \quad  \forall j \in \Nc,
\end{equation}
where $f \in \dot H^1(\G)$ and $g \in L^2(\G)$. 

More precisely, a (strong) solution of \eqref{eq:wave}--\eqref{eq:wave-CI} is a continuous function $u : I \to \tilde H^1(\G) \cap \dot H^2(\G^*)$ with continuous derivative $\partial_t u : I \to H^1(\G)$ and continuous second derivative $\partial_{tt} u : I \to L^2(\G)$, such that \eqref{eq:wave} holds in $L^2(\G)$ for all $t \in I$, with \eqref{eq:wave-continuite}--\eqref{eq:wave-Robin} for all $t \in I$, and with \eqref{eq:wave-CI} in $\tilde H^1(\G) \times L^2(\G)$.

We set $\HH = \tilde  H^1(\G) \times L^2(\G)$ and
\begin{equation} \label{Dom-Wc}
\Dom(\Wc_\a) = \set{U = (u,v) \in \big( \tilde H^1(\G) \cap \dot H^2(\G^*) \big) \times H^1(\G) \, : \, \sum_{j=1}^N u_j'(0) + \a v(0) = 0}.
\end{equation}
Then we define the unbounded operator $\Wc_\a$ in the Hilbert space $\HH$ by
\begin{equation} \label{def:Wc}
\forall U = (u,v) \in \Dom(\Wc_\a), \quad \Wc_\a \begin{pmatrix} u \\ v \end{pmatrix} = 
\begin{pmatrix} 
v \\
u''
\end{pmatrix}.
\end{equation}
We endow $\Dom(\Wc_\a)$ with the graph norm:
\[
\nr{U}_{\Dom(\Wc_\a)}^2 = \nr{\Wc_\a U}_\HH^2 + \nr{U}_{\HH}^2, \quad \forall U \in \Dom(\Wc_\a). 
\]
Then we see that $u$ is a solution of \eqref{eq:wave}--\eqref{eq:wave-CI} on $I$ if and only if $U = (u,\partial_t u)$ belongs to $C^1(\R_+,\HH) \cap C^0(\R_+,\Dom(\Wc))$ and satisfies
\begin{equation} \label{eq:wave-Wc}
\begin{cases}
U'(t) = \Wc U(t), \quad \forall t \in I,\\
U(0) = (f,g).
\end{cases}
\end{equation}

Our purpose in this paper is to describe the spectrum of the operator $\Wc_\a$ and the behaviour of the evolution problem \eqref{eq:wave}--\eqref{eq:wave-CI} (or equivalently \eqref{eq:wave-Wc}). The first step is to check that $\Wc_\a$ is closed with non-empty resolvent set. In Section \ref{Sec.general} we prove the following more precise result.
\begin{theorem} \label{th:Wa-max-diss}
The operator $\Wc_\a$ is maximal accretive (respectively maximal dissipative, respectively skew-adjoint) if $\Re(\a) \geq 0$ (respectively $\Re(\a) \leq 0$, respectively $\Re(\a) = 0$).
\end{theorem}

An alternative approach, based on the dominant Schur complement, 
to introduce the wave operator with possibly highly 
irregular dampings (possibly distributions)
has been recently developed by Gerhat~\cite{Gerhat}.

\subsection{The Robin vertex condition}

The definition of $\Dom(\Wc_\a)$ contains continuity at 0 for $u$ and $v$, and the additional condition will be referred to as the Robin condition. 

Notice that for $N = 1$, the graph $\G$ reduces to the half-line $]0,+\infty[$ and in this case we recover the usual Robin condition $u'(0) + \a v(0) =0$ at the boundary. When $N = 2$ we can identify the two edges with $]0,+\infty[$ and $]-\infty,0[$ 
(by considering the transformation $x \mapsto u_2(-x)$ on $]-\infty,0[$), 
so $\G$ is identified with $\R$ and in this setting the Robin condition yields the usual jump condition
$
u'(0^+) - u'(0^-) = -\a v(0).
$

Notice also that the two extreme situations $\alpha=0$ and $\alpha=\infty$
correspond, respectively, to Neumann (or Kirchhoff)
and Dirichlet boundary conditions 
imposed at the central vertex.

The motivation for considering the vertex condition \eqref{eq:wave-Robin} for our wave equation is that it corresponds to a singular damping 
localised \emph{at} the vertex. Formally, \eqref{eq:wave} and \eqref{eq:wave-Robin} mean
\[
\partial_{tt} u - \partial_{xx} u - \a \d \partial_t u = 0,
\]
where $\d$ is a generalisation of the Dirac distribution on $\G$. 
Similarly, for the corresponding wave operator we can formally write 
\[
\Wc_\a \begin{pmatrix} u \\ v \end{pmatrix} = \begin{pmatrix} v \\ u'' + \a \d \end{pmatrix}.
\]

This interpretation is supported by writing the quadratic form associated to the operator $\Wc_\a$ defined by \eqref{Dom-Wc}--\eqref{def:Wc}. Indeed, for $U = (u,v) \in \Dom(\Wc_\a)$ we have 
\begin{equation} \label{eq:fquad}
\langle \Wc_\a U, U \rangle_\HH = \langle v',u'\rangle_{L^2(\G)} + \langle u'' , v \rangle_{L^2(\G)} = 2i\Im \langle v',u' \rangle_{L^2(\G)} + \a \abs{v(0)}^2.
\end{equation}

To go further, we compare our operator $\Wc_\a$ 
with an operator with a damping 
on the edges localised \emph{near} the vertex.

For $j \in \Nc$ we consider $\rho_j \in L^1(\R_+^*,\R_+)$ such that $\int_0^{+\infty} \sqrt x \rho_j(x) \, dx < +\infty$. We assume that 
$
\sum_{j=1}^N \int_0^{+\infty} \rho_j(x) \, dx = 1.
$
For $n \in \N^*$ and $j \in \Nc$ we set $\rho_{j}^n(x) = n \rho_j(xn)$. Then we set $\rho^n = (\rho_j^n)_{1\leq j \leq N}$.

For $n \in \N^*$ we consider on $\HH$ the operator 
\[
\Wc_{\a,n} = 
\begin{pmatrix}
0 & 1 \\
\partial^2 & \a \rho^n
\end{pmatrix},
\]
defined on the domain
\[
\Dom(\Wc_{\a,n}) = \set{U = (u,v) \in \big(\tilde H^1(\G) \cap \dot H^2(\G^*) \big) \times H^1(\G) \, : \, \sum_{j=1}^N u_j'(0) = 0}.
\]
A result similar to Theorem \ref{th:Wa-max-diss} holds for $\Wc_{\a,n}$ (see Proposition \ref{prop:Wcan-diss} below). In particular, setting
\[
\C_\pm = \set{z \in \C \, : \, \pm \Re(z) > 0},
\]
we deduce that if $\a \in \overline{\C_\pm}$ then for $z \in \C_\mp$ we have $z \in \rho(\Wc_\a)$. 
In Section \ref{Sec.limit}
we show that $\Wc_\a$ is the limit of $\Wc_{\a,n}$ in the sense of the norm of the resolvent.

\begin{theorem} \label{th:conv-norm-res}
Let $\a \in \overline{\C_\pm}$ and $z \in \C_\mp$. Then 
\[
\nr{(W_\a-z)\inv - (W_{\a,n}-z)\inv }_{\Lc(\HH)} \limt n {+\infty} 0.
\]
\end{theorem}

\subsection{Spectral properties of the wave operator}

Having shown that~$\Wc_\a$ is a well defined operator and a suitable model for the damped wave equation with damping at the central vertex,
we can now turn to its spectral properties.

We first notice that, as a consequence of Theorem \ref{th:Wa-max-diss}, if $\a \in \overline{\C_\pm}$ then $\C_\mp \subset \rho(\Wc_\a)$ and for $z \in \C_\mp$ we have 
\[
\nr{(\Wc_\a -z)\inv}_{\Lc(\HH)} \leq \frac 1 {\mp \Re(z)}.
\]

The main result of this paper is related to spectral properties on the other half-plane. There is no general theory for resolvent estimates inside the numerical range, but explicit computations provide a precise description of the spectrum and the resolvent for this particular problem.

\begin{theorem} \label{th:spectrum}
The spectrum of $\Wc_\a$ is 
\[
\begin{cases}
i\R & \text{if } \a \in \C \setminus \{\pm N\},\\
\overline{\C_+} & \text{if } \a = N,\\
\overline{\C_-} & \text{if } \a = -N.
\end{cases}
\]
Moreover, 
\begin{enumerate} [\rm(i)]
\item $i\R$ contains no eigenvalue nor residual spectrum of $\Wc_\a$,
\item if $\a = \pm N$ then any $z \in \C_\pm$ is an eigenvalue of $\Wc_\a$ of geometric multiplicity 1 and infinite algebraic multiplicity,
\item there exist $c_0, C > 0$ such that for $\a  \in \C_\pm \setminus \set{\pm N}$ and $z \in \C_\pm$ we have
\[
\max \left( \frac 1 {\abs {\Re(z)}}, \frac {c_0} {\abs z \abs {\a \mp N}} \right) \leq \nr{(\Wc_\a - z)\inv}_{\Lc(\HH)} \leq \frac {C}{\abs{\Re(z)}} \left( 1 + \frac 1{ \abs{\a \mp N}} \right).
\]
\end{enumerate}
\end{theorem}

\subsection{The damped wave equation}

Finally we go back to the time-dependant problem \eqref{eq:wave}--\eqref{eq:wave-CI}, or equivalently \eqref{eq:wave-Wc}. If $\Re(\a) \leq 0$, the operator $\Wc_\a$ generates by Theorem~\ref{th:Wa-max-diss} a contractions semigroup, so the problem \eqref{eq:wave-Wc} is well posed on $\R_+$. 
Moreover the energy 
$$
E(u;t) = \nr{\partial_t u(t)}_{L^2(\G)}^2 + \nr{\partial_x u(t)}_{L^2(\G)}^2 
$$
of the solution~$u$ is non-increasing.

In this paragraph we address the question of well-posedness and growth of the energy for \eqref{eq:wave-Wc} (or equivalently for \eqref{eq:wave}--\eqref{eq:wave-CI}) when $\Re(\a) > 0$. We have similar results for negative times when $\Re(\a) < 0$.

\begin{theorem} \label{th:damped-wave}
Let $\a \in \C_+$ and $(f,g) \in \Dom(\Wc_\a)$.
\begin{enumerate} [\rm(i)]
\item Assume that $\a \neq N$. The problem \eqref{eq:wave}--\eqref{eq:wave-CI} has a unique solution $u$ on $\R_+$. Moreover there exists $C > 0$ independent of $\a$ 
and $(f,g)$ such that for $t \geq 0$ we have
\begin{equation} \label{eq:energy}
E(u;t) \leq C \left( 1 + \frac 1 {\abs {\a - N}^2} \right) E(u;0).
\end{equation}

\item Assume that $\a = N$. Let 
\[
t_0 = \sup \set{t \geq 0 \, : \, \sum_{j=1}^N \big( f_j'(s) + g_j(s) \big) = 0 , \forall s \in [0,t] } \quad \in [0,+\infty].
\]
If $t_0 > 0$ then \eqref{eq:wave}--\eqref{eq:wave-CI} has an infinite number of solutions on $[0,t_0[$. In particular there exists a solution $u$ such that 
\[
E(u;t) \limt {t}{t_0} + \infty.
\]
If $t_0$ is finite then for any $\e > 0$ the problem \eqref{eq:wave}--\eqref{eq:wave-CI} has no solution on $[0,t_0+\e[$.  
\end{enumerate}
\end{theorem}

\section{General properties of the wave operator}\label{Sec.general}

In this section we prove some basic properties for the wave operator $\Wc_\a$. In particular, we give an expression for its resolvent with the spectral parameter lying in the suitable half-place (depending on the sign of $\Re(\a)$) and deduce that $\Wc_\a$ is maximal accretive and/or maximal dissipative (Theorem \ref{th:Wa-max-diss}). Notice that our proofs are quite robust and could be applied for the wave equation on more general graphs.

To prepare the proof of Theorem \ref{th:conv-norm-res} in the next section, we proceed with the same analysis for $\Wc_{\a,n}$, $n \in \N^*$.

We first record the following direct consequence of formula~\eqref{eq:fquad}.
 
\begin{proposition} \label{prop:Wa-diss}
The operator $\Wc_\a$ is accretive (respectively dissipative, respectively skew-symmetric) if $\Re(\a) \geq 0$ (respectively $\Re(\a) \leq 0$, respectively $\Re(\a) = 0$). In particular, if $\a \in \overline{\C_\pm}$ and $z \in \C_{\mp}$ then $(\Wc_\a-z)$ is injective with closed range. For $n \in \N^*$ the operator $\Wc_{\a,n}$ has the same properties.
\end{proposition}

Next we mention the following symmetry result.

\begin{proposition} \label{prop:adjoint}
For $\a \in \C$ we have $\Wc_\a^* = -\Wc_{-\bar \a}$.
\end{proposition}

\begin{proof}
For $U \in \Dom(\Wc_a)$ and $\tilde U \in \Dom(\Wc_{-\bar \a})$ we can check by direct computation that 
\[
\big\langle \Wc_\a U, \tilde U \big\rangle_\HH = - \big\langle U , \Wc_{-\bar \a} \tilde U \big\rangle_\HH.
\]
This proves that $\Dom(\Wc_{-\bar \a}) \subset \Dom(\Wc_\a^*)$ and that $\Wc_\a^* = -\Wc_{-\bar \a}$ on $\Dom(\Wc_{-\bar \a})$. 

Now let $\tilde U = (\tilde u, \tilde v) \in \Dom(\Wc_\a^*)$ and $F =  (f,g) =\Wc_\a^* \tilde U \in \HH$. For all $U = (u,v) \in \Dom(\Wc_\a)$ we have $\big \langle \Wc_\a U , \tilde U \big \rangle _\HH = \innp U F_\HH$, which gives
\begin{equation} \label{eq:u'f'vg}
\innp{v'}{\tilde u'}_{L^2(\G)} + \innp{u''}{\tilde v}_{L^2(\G)} = \innp{u'}{f'}_{L^2(\G)} + \innp{v}{g}_{L^2(\G)}.
\end{equation}
Let $j \in \Nc$. Applied with $u = 0$, $v_j \in C_0^\infty(\R_+^*)$ and $v_k = 0$ for $k \neq j$, this proves that $\tilde u_j' \in H^1(\R_+^*)$ and $\tilde u_j'' = -g_j$. Applied with $v= 0$, $u_j \in C_0^\infty(\R_+^*)$ and $u_k = 0$ for $k \neq j$, we deduce that there exists a constant $\b_j$ such that $\tilde v_j' = - f_j' + \b_j$ in the sense of distributions. Since $f_j'$ and $\tilde v_j$ are in $L^2(\R_+^*)$, we necessarily have $\b_j = 0$, so $\tilde v_j ' = f_j'$ belongs to $L^2(\R_+^*)$. 

We can rewrite \eqref{eq:u'f'vg} as
\begin{multline*}
\innp{u'}{f'}_{L^2(\G)} + \innp{v}{g}_{L^2(\G)} \\
=
- \innp{v}{\tilde u''}_{L^2(\G)} - v(0) \sum_{j=1}^N \overline{\tilde u_j'(0)}
- \innp{u'}{\tilde v'}_{L^2(\G)} - \sum_{j=1}^N u_j'(0) \overline{\tilde v(0)}.
\end{multline*}
This gives 
\[
- v(0) \sum_{j=1}^N \overline{\tilde u_j'(0)} +  \a  v(0) \overline{\tilde v(0)} = 0,
\]
which implies that $\sum_{j=1}^N \tilde u_j'(0) - \bar \a \tilde v(0) = 0$. Then $\Dom(\Wc_\a) \subset \Dom(\Wc_{-\bar \a})$, and the proof is complete. 
\end{proof}

In Proposition \ref{prop:res-Wca} below, 
we will give an expression for the resolvent of $\Wc_\a$ valid in $H^1(\G) \times L^2(\G)$. This is a dense subset of $\HH$ by the following classical lemma.

\begin{lemma}
$H^1(\G)$ is dense in $\dot H^1(\G)$.
\end{lemma}

\begin{proof}
Let $u = (u_j)_{j \in \Nc} \in \dot H^1(\G)$. Let $j \in \Nc$. Let $\chi \in C^\infty(\R_+,[0,1])$ be equal to 1 on $[0,1]$ and equal to 0 on $[2,+\infty[$. For $R \geq 1$ and $x \geq 0$ we set $\chi_R(x) = \chi \big( \frac x R \big)$. We have 
\begin{align*}
 \nr{((1-\chi_R) u_j)'}_{L^2(\R_+^*)} \leq \nr{(1-\chi_R)u_j'}_{L^2(\R_+^*)} + \nr{\chi_R' u_j}_{L^2(\R_+^*)}.
\end{align*}
The first term goes to 0 as $R\to\infty$ by the dominated convergence theorem. 
We estimate the second term. Let $\e > 0$. There exists $C_j > 0$ such that for all $x \geq 1$ we have $\abs{u_j(x)} \leq C_j \sqrt x$. Let $x_0 >0$ be so large that $2 \sqrt 2 C_j \nr{u'}_{L^2(x_0,\infty)} \leq \e$, and let $R \geq 1$ be so large that $\abs{u(x_0)}^2 \leq \frac {\e R}2$. Then for $x \geq \max(x_0,R)$ we have  
\begin{multline*}
\abs{u_j(x)}_{L^2(\R_+^*)}^2 \leq \abs{u_j(x_0)}^2 + 2 \int_{x_0}^x \abs{u_j(s)} \abs{u_j'(s)} \, ds\\
\leq \abs {u_j(x_0)}^2 + \sqrt 2 C_j x \nr{u_j'}_{L^2(x_0,\infty)} \leq \e x,
\end{multline*}
so
\begin{align*}
\nr{\chi_R' u_j}^2 
\leq \frac {\nr{\chi'}_\infty^2}{R^2} \int_R^{2R} \abs{u_j(x)}^2 \, dx 
\leq 2 \e \nr{\chi'}_\infty^2.
\end{align*}
The conclusion follows.
\end{proof}

We denote by $H^{-1}(\G)$ the space of continuous semilinear forms on $H^1(\G)$ (we have $\ell (\f_1+\b \f_2) = \ell(\f_1) + \bar \b \ell(\f_2)$ for $\ell \in H\inv(\G)$, $\f_1,\f_2 \in H^1(\G)$ and $\b \in \C$). In particular, $\d : \phi \mapsto \overline{\phi(0)}$ belongs to $H\inv(\G)$. We refer to \cite[pp.~3--4]{Edmunds-Evans} for a discussion about this choice.

For $\a \in \C$, $z \in \C$ and $n \in \N^*$ 
we define bounded operators $Q_\a(z)$ and $Q_{\a,n}(z)$ in $\Lc(H^1(\G),H\inv(\G))$ by
\[
\innp{Q_\a(z) \psi}{\f}_{H\inv(\G),H^1(\G)} = \innp{\psi'}{\f'}_{L^2(\G)} - \a z \psi(0) \overline{\f(0)} + z^2 \innp{\psi}{\f}_{L^2(\G)}
\]
and 
\[
\innp{Q_{\a,n}(z) \psi}{\f}_{H\inv(\G),H^1(\G)} = \innp{\psi'}{\f'}_{L^2(\G)} - \a z \innp{\rho^n \psi}{\f}_{L^2(\G)} + z^2 \innp{\psi}{\f}_{L^2(\G)},
\]
for all $\psi , \f \in H^1(\G)$ (the scalar products are linear on the left and semilinear on the right).

\begin{proposition} \label{prop:Raz}
Let $\a \in \overline{\C_\pm}$, $z \in \C_\mp$ and $n \in \N^*$. Then $Q_\a(z)$ and $Q_{\a,n}(z)$ are invertible. Moreover, the norm $\nr{Q_{\a,n}(z)\inv}_{\Lc(H\inv(\G),H^1(\G))}$ is bounded uniformly in $n \in \N^*$.
\end{proposition}

\begin{proof}
We consider the case $\Re(\a) \geq 0$ and $\Re(z) < 0$ (the case $\Re(\a) \leq 0$ and $\Re(z) > 0$ is similar). Let $\th = \arg(z) - \pi \in \big]-\frac \pi 2 , \frac \pi 2 \big[$ and $\eta = \arg(\a) \in \big[-\frac \pi 2 , \frac \pi 2 \big]$. For $w \in H^1(\G)$ we have
\begin{eqnarray*}
\lefteqn{\Re \big(e^{-i\th} \innp{Q_{\a,n}(z) w}{w}_{H\inv(\G),H^1(\G)} \big)}\\
&& = \cos(\th) \nr{w'}_{L^2(\G)}^2 + \cos(\eta) \abs {\a z} \innp{\rho^n w}{w}_{L^2(\G)} + \cos(\th) \abs{z}^2 \nr{w}_{L^2(\G)}^2\\
&& \geq \min(1,\abs z^2) \cos(\th) \nr{w}_{H^1(\G)}^2.
\end{eqnarray*}
By the Lax--Milgram Theorem, $Q_{n,\a}(z)$ is invertible and $\nr{Q_{\a,n}(z)\inv}_{\Lc(H\inv(\G),H^1(\G))}$ is uniformly bounded in $n \in \N^*$. 
We proceed similarly for $Q_\a(z)$.
\end{proof}

We set $R_{\a}(z) = Q_\a(z)\inv$ and $R_{\a,n}(z) = Q_{\a,n}(z)\inv$.

\begin{proposition} \label{prop:Raz-H2}
Let $\a \in \overline{\C_\pm}$, $z \in \C_\mp$ and $n \in \N^*$. Let $h \in L^2(\O)$ and $\k \in \C$.
\begin{enumerate}[\rm(i)]
\item Let $w = R_\a(z) (h + \k \d) \in H^1(\G)$. Then for $j \in \Nc$ we have $w_j''  = z^2 w_j - h_j \in L^2(\R_+^*)$, and moreover 
\begin{equation} \label{eq:vertex-cond-kappa}
\sum_{j=1}^N w_j'(0) + \a w(0) = - \k.
\end{equation}
\item Let $w = R_{\a,n}(z) (h + \k \rho^n) \in H^1(\G)$. Then $\sum_{j \in \Nc} w'(0) = 0$ and for $j \in \Nc$ we have $w_j''  = - \a z \rho^n w + z^2 w_j - h_j - \k \rho^n \in L^2(\R_+^*)$.
\end{enumerate}
\end{proposition}

\begin{proof}
Since $h + \k \d \in H\inv(\G)$, $w$ is well defined as an element of $H^1(\G)$ 
by Proposition~\ref{prop:Raz}.
For all $\phi \in H^1(\G)$ we have 
\begin{equation} \label{eq:wh}
\innp{w'}{\phi'}_{L^2(\G)} - \a z w(0) \overline{\phi(0)} + z^2 \innp{w}{\phi}_{L^2(\G)} = \innp{h}{\phi}_{L^2(\G)} + \k \overline{\phi(0)}.
\end{equation}
As in the proof of Proposition \ref{prop:adjoint}, by choosing $\phi$ supported away from the vertex we see that for all $j \in \Nc$ we have in the sense of distributions $w_j'' = z^2 w_j - h_j$. In particular $w_j \in H^2(\R_+^*)$. Then, after integrations by parts in \eqref{eq:wh} we get 
\[
-\sum_{j=1}^N w_j'(0) \overline{\phi(0)} - z \a w(0) \overline{\phi(0)} = \k \overline{\phi(0)}.
\]
This gives \eqref{eq:vertex-cond-kappa}. The second statement is similar.
\end{proof}

We define $\partial^2 \in \Lc(\dot H^1(\G) , H\inv(\G))$ by $\innp{\partial^2 \psi}{\f}_{H\inv(\G),H^1(\G)} = - \innp{\psi'}{\f'}_{L^2(\G)}$ for all $\psi,\f \in H^1(\G)$. In particular we have $Q_\a(z) = - \partial^2 - \a z \d + z^2$.

\begin{proposition} \label{prop:res-Wca}
Let $\a \in \overline{\C_\pm}$ and $z \in \C_\mp$. 
We have $z \in \rho(W_\a)$ and 
\begin{equation} \label{eq:Rca}
(\Wc_\a - z)\inv  = 
\begin{pmatrix}
- z\inv \big( R_{\a}(z) \partial^2 + 1 \big) & - R_{\a}(z) \\
- R_{\a}(z) \partial^2 & - z R_{\a}(z)
\end{pmatrix}.
\end{equation}
Moreover, for $F = H^1(\G) \times L^2(\G)$ we also have
\begin{equation} \label{eq:Rca-F}
(\Wc_{\a}-z)\inv F 
= 
\begin{pmatrix}
R_{\a}(z) (\a \d - z) & - R_{\a}(z) \\
1 + R_{\a}(z) (z \a \d - z^2) & - z R_{\a}(z)
\end{pmatrix} F.
\end{equation}
For $n \in \N^*$ have the same results with $\Wc_\a$, $R_\a(z)$ and $\d$ replaced by $\Wc_{\a,n}$, $R_{\a,n}(z)$ and~$\rho^n$.
\end{proposition}

\begin{proof}
For $F \in H^1(\G) \times L^2(\G)$ we denote by $\Rc_\a(z)F$ the right-hand side of \eqref{eq:Rca-F}. We set $F = (f,g)$ and $U = \Rc_\a(z) F = (u,v)$. We have $u,v \in H^1(\G)$. By Proposition \ref{prop:Raz-H2}, we have $u_j'' \in L^2(\R_+^*)$,  $u_j'' = z^2 u + zf + g$ and $v_j = f_j + zu_j$ for all $j \in \Nc$. On the other hand we have at the vertex:
\[
\sum_{j=1}^N u_j'(0) + \a z u(0) = - \a f(0).
\]
This gives the Robin condition. All this proves that $U \in \Dom(\Wc_\a)$ and $(\Wc_\a-z) U = F$. In particular $(\Wc_\a - z)$ has dense range. Since $(\Wc_\a-z)$ is injective with closed range by Proposition \ref{prop:Wa-diss}, $z$ belongs to $\rho(\Wc_\a)$ and $(\Wc_\a-z)\inv = \Rc_\a(z)$ on $H^1(\G) \times L^2(\G)$.

Now we denote by $\tilde \Rc_\a(z)$ the right-hand side of \eqref{eq:Rca}. From the properties of $R_\a(z)$ we see that $\tilde \Rc_\a(z)$ defines a bounded operator on $\HH$. Moreover, for $F \in H^1(\G) \times L^2(\G)$ we have $\tilde \Rc_\a(z) F = \Rc_\a(z) F = (\Wc_\a-z)\inv F$. Since $H^1(\G) \times L^2(\G)$ is dense in $\HH$, this proves that $\tilde \Rc_\a(z) = (\Wc_\a-z)\inv$.

The proof for $\Wc_{\a,n}$ is similar. We omit the details.
\end{proof}

With Proposition \ref{prop:res-Wca} we can complete the statement of Proposition \ref{prop:Wa-diss}. This gives in particular Theorem \ref{th:Wa-max-diss}.

\begin{proposition} \label{prop:Wcan-diss}
The operator $\Wc_{\a}$ is maximal accretive (respectively maximal dissipative, respectively skew-adjoint) if $\Re(\a) \geq 0$ (respectively $\Re(\a) \leq 0$, respectively $\Re(\a) = 0$). For $n \in \N^*$ the operator $\Wc_{\a,n}$ has the same properties.
\end{proposition}

\section{Damping at the vertex as a limit model for damping on the edges}\label{Sec.limit}

In this section we establish Theorem~\ref{th:conv-norm-res}. 
We first check that the sequence $(\rho^n)_{n \in \N^*}$ is an approximation of the Dirac distribution.

\begin{lemma} \label{lem:rhon-delta}
We have 
\[
\nr{\rho^n - \d}_{\Lc(H^1(\G),H\inv(\G))} \limt n {+\infty} 0.
\]
\end{lemma}

\begin{proof}
Let $u,w \in H^1(\G)$. For $n \in \N^*$ we have
\[
\innp{(\rho^n - \d)u}{w}_{H\inv(\G),H^1(\G)} = \sum_{j=0}^N \int_0^{+\infty} \rho_j^n(x) \big( u_j(x) \overline{w_j(x)} - u(0) \overline{w(0)} \big) \, dx .
\]
For $j \in \Nc$ we have  
\begin{align*}
\int_0^{+\infty} \rho_j^n(x) \abs{(u_j \overline{w_j})(x) - (u \overline w)(0)} \, dx 
& \leq \int_{0}^{+\infty} \rho_j^n(x) \abs{\int_0^x (u_j \overline{w_j})'(s) \,ds } \, dx\\
& \leq \int_0^{+\infty} \rho_j^n(x) \sqrt x \nr{(u_j \overline{w_j})'}_{L^2(\R_+^*)} \, dx \\
& \leq \frac { \|(u_j \overline{w_j})'\|_{L^2(\R_+^*)}}{\sqrt n} \int_0^{+\infty} \rho_j(y) \sqrt y \, dy.
\end{align*}
Then there exists $c>0$ such that for all $n \in \N^*$ and $u,w \in H^1(\G)$ we have  
\[
\abs{\innp{(\rho^n - \d)u}{w}_{H\inv(\G),H^1(\G)}} \leq \frac c {\sqrt n} \nr{u}_{H^1(\G)} \nr{w}_{H^1(\G)}.
\]
This concludes the proof of the lemma.
\end{proof}

\begin{proposition} \label{prop:approx-Ra}
Let $\a \in \overline{\C_\pm}$ and $z \in \C_\mp$. We have 
\[
\nr{R_\a(z) - R_{\a,n}(z)}_{\Lc(H\inv(\G),H^1(\G)} \limt n {+\infty} 0.
\]
\end{proposition}

\begin{proof}
The resolvent identity gives 
\[
R_\a(z) - R_{\a,n}(z) =  z R_\a(z) \big( \a \d - \a \rho^n \big) R_{\a,n}(z).
\]
Since the size of $R_{\a,n}(z)$ in $\Lc(H\inv(\G),H^1(\G))$ 
does not depend on $n \in \N^*$
(recall Proposition \ref{prop:Raz-H2}), 
we conclude with help of Lemma \ref{lem:rhon-delta}.
\end{proof}

Now we are in a position to establish Theorem~\ref{th:conv-norm-res}.
\begin{proof}[Proof of Theorem~\ref{th:conv-norm-res}]
For $F \in H^1(\G) \times L^2(\G)$ we have 
\begin{eqnarray*}
\lefteqn{\nr{\big( (W_\a-z)\inv - (W_{\a,n}-z)\inv \big) F}_\HH}\\
&& \lesssim \nr{R_{\a,n}(z) - R_\a(z)}_{\Lc(H\inv(\G),H^1(\G))} \big( \nr{f''}_{H\inv(\G)} + \nr{g}_{L^2(\G)} \big)\\
&& \lesssim \nr{R_{\a,n}(z) - R_\a(z)}_{\Lc(H\inv(\G),H^1(\G))}\nr{F}_\HH.
\end{eqnarray*}
Here the relation $f \lesssim g$ means that 
there exists a constant~$C$ (independent of~$n$) 
such that $f \leq C g$. 
By density of $H^1(\G) \times L^2(\G)$ in $\HH$ we get 
\[
\nr{(W_\a-z)\inv - (W_{\a,n}-z)\inv }_{\Lc(\HH)} \lesssim \nr{R_{\a,n}(z) - R_\a(z)}_{\Lc(H\inv(\G),H^1(\G))},
\]
and we conclude with Proposition \ref{prop:approx-Ra}.
\end{proof}

\section{Spectrum of the wave operator}\label{Sec.proofs}
%

In this section we prove Theorem \ref{th:spectrum}. By Proposition \ref{prop:adjoint} it is enough to consider the case $\Re(\a) \geq 0$. In this case, we already know by Proposition \ref{prop:res-Wca} that $\C_-  \subset \rho (\Wc_\a)$. We use explicit computation to describe the spectrum on the right half-plane. Then Theorem \ref{th:spectrum} follows from Propositions \ref{prop:spectrum-iR}, \ref{prop:eigenvalues}, \ref{prop:minor-norm} and \ref{prop:spectrum-right} below.

Given a closed operator~$\mathcal{W}$ in a Hilbert space~$\mathscr{H}$,
we denote its \emph{point spectrum} (i.e.~the set of eigenvalues~$\mathcal{W}$) 
by $\sigma_{\mathrm{p}}(\mathcal{W})$.  
One says that $\lambda \in \sigma(\mathcal{W})$ 
belongs to the \emph{continuous spectrum} $\sigma_{\mathrm{c}}(\mathcal{W})$
(respectively, \emph{residual spectrum} $\sigma_{\mathrm{r}}(\mathcal{W})$) 
of~$\mathcal{W}$ 
if $\lambda\not\in\sigma_\mathrm{p}(\mathcal{W})$
and the closure of the range of the shifted operator
$\mathcal{W}-\lambda$ equals~$\mathscr{H}$  
(respectively, the closure is a proper subset of~$\mathscr{H}$).

We first consider the spectrum on the imaginary axis.

\begin{proposition} \label{prop:spectrum-iR}
Let $\a \in \C$. 
Then $i\R \subset \sigma_\mathrm{c}(\Wc_\a)$.%
\end{proposition}

\begin{proof}
Let $\theta \in \R$.
Let $\phi \in C_0^\infty(\R_+^*)$ be supported in [1,2] and such that $\nr{\phi}_{L^2(\R_+^*)} = 1$. For $n \in \N^*$ we define $u_n = (u_{n,j})_{j \in \Nc}$ by
\[
u_{n,1}(x) = \frac {e^{i\th x}}{\sqrt n} \phi \left( \frac x n \right), \quad u_{n,j}(x) = 0, \quad  j \in \{ 2,\dots,N\}, x > 0.
\]
Then we set $U_n = (u_n,i\th u_n) \in \Dom(\Wc_\a)$. For $n \in \N^*$ we have $\nr{u_n'}_{L^2(\G)} = \abs \th + O(n\inv)$, $\nr{i\th u_n}_{L^2(\G)} = \abs \th$ and $\nr{u_n'' + \th^2 u_n}_{L^2(\G)} = O(n\inv)$, so 
\[
\nr{U_n}_\HH^2 = 2 \th^2 +O(n^{-1}) \quad \text{and} \quad \nr{(\Wc_\a - i\theta)U_n}_\HH^2 = \nr{u_n'' + \th^2 u_n}_{L^2(\R)}^2 = O(n^{-2}). 
\]
This proves that $i\th \in \Sp(\Wc_\a)$.

Now let $U = (u,v) \in \Dom(\Wc_\a)$ such that $\Wc_\a U = i \th U$. 
That is, $v = i\th u$ and $u_j'' = -\th^2 u_j$ on $\R_+$ for all $j \in \Nc$. Since $u_j \in \dot H^1(\R_+)$, this implies that $u_j = 0$. 
Then $U = 0$, so $i\th$ cannot be an eigenvalue of $\Wc_\a$.

Finally, assume that $i\theta \in \sigma_{\mathrm{r}}(\Wc_\a)$.
Then $-i\theta \in \sigma_{\mathrm{p}}(\Wc_\a^*)$.
By Proposition~\ref{prop:adjoint}, 
$i\theta \in \sigma_{\mathrm{p}}(\Wc_{-\bar\a})$.
However, the existence of eigenvalues on the entire imaginary
axis has been already excluded.
\end{proof}

Next we show that in the particular case $\a = N$ 
the right-half plane is filled with eigenvalues.

\begin{proposition} \label{prop:eigenvalues}
Any $z \in \C_+$ is an eigenvalue of $\Wc_N$ with geometric multiplicity 1 and infinite algebraic multiplicity.
\end{proposition}

\begin{proof}
Let $z \in \C_+$. Assume that $U = (u,v) \in \Dom(\Wc_N)$ is such that $\Wc_N U = z U$. Then $v = zu$ and $u'' = z v = z^2 u$. Let $j \in \Nc$. Since $u_j \in L^2(\R_+^*)$, there exists $A_j \in \C$ such that $u_j(x) = A_j e^{-z x}$. By continuity at 0, the coefficients $A_j$ for $j \in \Nc$ are all equal. Thus $U$ is proportional to the vector $U_1 = (u_1,v_1)$ defined by 
\[
\begin{cases}
u_{1,j}(x) = e^{-zx},\\
v_{1,j}(x) = ze^{-zx}.
\end{cases}
\]
Notice in particular that $U_1$ is radial 
(the expressions above do not depend on $j \in \Nc$). 
Conversely, we can check that 
the vector $U_1$ defined in this way belongs to $\Dom(\Wc_N)$ and is an eigenvector corresponding to the eigenvalue $z$. This proves that $z$ is a geometrically simple eigenvalue of $\Wc_N$.

For $n \geq 2$ we define $u_n$ by 
\[
u_{n,j}(x) =  \frac {(-1)^{n-1} x^{n-1}}{(n-1)!} e^{-zx}, \quad j \in \Nc, x  > 0.
\]
At the same time, we define $v_n$ by 
\[
v_n = z u_n + u_{n-1}.
\]
It is straightforward to check that for all $n \geq 2$ 
we have $U_n = (u_n,v_n) \in \Dom(\Wc_N)$ and $\Wc_N U_n = z U_n + U_{n-1}$. This proves that $z$ has algebraic multiplicity $+\infty$ as an eigenvalue of~$\Wc_N$.
\end{proof}

The precedent proposition establishes part~(ii) of Theorem~\ref{th:spectrum}.
The other part that for $\a \in \C_+ \setminus \{N\}$ 
there is no spectrum in the right-half plane
will be proved in a moment.
First, however, let us argue that even if there is no spectrum,
the pseudospectra are highly non-trivial there
and actually explode as $\a \to N$.
This is quantified by obtaining the following resolvent estimate.

\begin{proposition} \label{prop:minor-norm}
There exists $c_0 > 0$ such that for $\a \in \C_+ \setminus \{ N\}$ 
and $z \in \C_+ \cap \rho(\Wc_\a)$ we have 
\[
\nr{(\Wc_\a-z)\inv}_{\Lc(\HH)} \geq \max\left( \frac 1 {\Re(z)}, \frac {c_0}{\abs z \abs{N-\a}} \right).
\]
\end{proposition}

\begin{proof}
For $j \in \Nc$ and $x > 0$ we set $\eta(x) = e^{- \frac {\a z x}N}$. Then we define $U = (u,v)$ by $u_j = \eta$ and $v_j = z \eta$ for all $j \in \Nc$. We have $U \in \Dom(\Wc_\a)$,
\[
\nr U_\HH^2 =  N \left(\frac {\abs \a^2}{N^2} + 1 \right) \abs z^2 \nr{\eta}_{L^2(\R_+^*)}^2
\]
and 
\[
\nr{(\Wc_\a-z) U}_\HH^2 = N \abs {\frac {\a^2}{N^2} - 1 }^2  \abs z^4 \nr{\eta}_{L^2(\R_+^*)}^2,
\]
so
\[
\nr{(\Wc_\a-z)U}_\HH^2  \leq  \abs z^2  \abs {\a - N}^2 \frac {\abs {\a + N}^2}{N^2(\abs \a^2 + N^2)} \nr{U}_\HH^2.
\]
This proves that 
\[
\nr{(\Wc_\a-z)\inv}_{\Lc(\HH)} \geq \frac {c_0}{\abs z \abs{\a - N}}, \quad \text{with} \quad c_0 = \inf_{\a \in \C_+} \frac {N\sqrt{\abs \a^2 + N^2}}{\abs {\a + N}} > 0.
\]
On the other hand we also have
\[
\nr{(\Wc_\a-z)\inv}_{\Lc(\HH)} \geq \frac 1 {\mathsf{dist}(z ,\Sp(\Wc_\a))} \geq \frac 1 {\Re(z)},
\]
and the conclusion follows.
\end{proof}

To show that for $\a \in \C_+ \setminus \{N\}$
there is no spectrum in the right-half plane,
we compute explicitely the resolvent of $\Wc_\a$. 
This will also give the upper bound for the norm of the resolvent.

For $z \in \C_+$ and $y \in \R$ we set
\begin{equation} \label{def:rho-z}
\rho_z(y) = \frac {z e^{-z \abs y}} 2.
\end{equation}
Notice that $\rho_z(y) = -z^2 G_z(y)$, where $G_z$ is the Green function for the Helmholtz equation in dimension 1. On a half-line (hence on a star graph), the solution of the Helmholtz equation is also given by convolution with $G_z$ or $\rho_z$. 

Let $h : [0,+\infty[ \to \C$ and $x \geq 0$. When this makes sense we set 
\begin{equation} \label{def:convolution-R+}
(\rho_z*h)(x) = \int_0^{+\infty} \rho_z(x-s) h(s) \, ds.
\end{equation}
We will use the following properties on this convolution product.

\begin{lemma} \label{lem:conv-dotH1}
Let $z \in \C_+$.

\begin{enumerate} [\rm(i)]
\item The convolution $(\rho_z*h)(x)$ is well defined for $h$ in $L^2(\R_+^*)$ or $h \in \dot H^1(\R_+^*)$ and any $x \geq 0$.

\item For $h \in L^2(\R_+^*)$ we have $(\rho_z*h) \in L^2(\R_+^*)$ and
\begin{equation} \label{eq:conv-L2}
\nr{\rho_z*h}_{L^2(\R_+^*)} \leq \frac {\abs z}{\Re(z)} \nr{h}_{L^2(\R_+^*)}.
\end{equation}

\item For $h$ in $L^2(\R_+^*)$ or $\dot H^1(\R_+^*)$ the convolution $(\rho_z*h)$ belongs to $\dot H^1(\R_+^*)$ and for almost all $x > 0$ we have 
\begin{equation} \label{eq:der-rho-h}
(\rho_z * h)'(x) = z (\rho_z * h) (x) - z^2 \int_0^x e^{-z(x-s)} h(s) \, ds.
\end{equation}
For $h \in \dot H^1(\R_+)$ this also gives 
\[
(\rho_z * h)(x) = \frac 1 z (\rho_z*h')(x) -  \int_0^x e^{-z(x-s)} h'(s)\, ds + h(x) - \frac {e^{-zx} h(0)}2
\]
and 
\[
(\rho_z * h)'(x) = (\rho * h')(x) + \frac z 2 e^{-zx} h(0).
\]

\item For $h \in L^2(\R_+^*)$ we have
\[
\abs{(\rho_z*h)(0)} \leq \frac {\abs z}{2 \sqrt{2 \Re(z)}} \nr{h}_{L^2(\R_+^*)}.
\]
\end{enumerate}
\end{lemma}

    \detail 
    {
    \[
    h(s) = e^{i \Im(z) s} \sqrt{\Re(z)} \1_{0,\Re(z)\inv]}
    \]
    \begin{align*}
    \int_0^{+\infty} e^{-z s} h(s) \, ds = \sqrt{\Re(z)} \int_0^{\frac 1 {\Re(z)}} e^{-\Re(z) s} \, ds = \frac 1 {\sqrt{\Re(z)}} \int_0^1 e^{-\tau} \, d\tau.
    \end{align*}
    }

\begin{proof}
Since $\rho_z$ decays exponentially, the integral \eqref{def:convolution-R+} is well defined for $h \in L^2(\R_+^*)$ or $h \in \dot H^1(\R_+)$ (in this case $h(s)$ grows at most like $\sqrt s$), and for any $x \geq 0$.

If $h \in L^2(\R_+^*)$ we extend $h$ by 0 on $]-\infty,0[$ to define a function $\tilde h$ on $\R$. Then $(\rho_z*h)$ is the restriction to $[0,+\infty[$ of the usual convolution $(\rho_z * \tilde h)$ on $\R$, and
\[
\nr{\rho_z * h}_{L^2(\R_+^*)} \leq \| \rho_z * \tilde h \| _{L^2(\R)} \leq \nr{\rho}_{L^1(\R)} \| \tilde h \|_{L^2(\R)} = \frac {\abs z}{\Re(z)} \nr{h}_{L^2(\R_+^*)}.
\]
        \detail{
        For $\phi \in C_0^\infty(\R_+^*)$ we have 
        \begin{align*}
        - \int_0^{+\infty} (\rho_z * h)(x) \phi'(x) \, dx
        & = - \int_0^{+\infty} \int_0^{+\infty} \rho_z(x-s) h(s) \phi'(x) \, ds \, dx\\
        & = \int_0^{+\infty} \int_0^{+\infty} \rho_z'(x-s) h(s) \phi(x) \, ds \, dx\\
        & = \int_0^{+\infty} (\rho_z' * h)(x) \phi(x) \, dx
        \end{align*}
        }

The third statement is straightforward computation and the last property follows from the Cauchy--Schwarz inequality 
\[
\abs{(\rho_z *h)(0)} =  \abs{\int_0^{+\infty} \rho_z(-s) h(s) \, ds } \leq \nr{\rho_z}_{L^2(\R_-)} \nr{h}_{L^2(\R_+^*)} = \frac {\abs{z}}{2 \sqrt{2\Re(z)}}\nr{h}_{L^2(\R_+^*)}. \qedhere
\]
\end{proof}

Now we are in a position to complete the proof of Theorem \ref{th:spectrum}.

\begin{proposition} \label{prop:spectrum-right}
Assume that $\a \in \C_{+} \setminus \{  N \}$. Then $\C_+ \subset \rho(\Wc_\a)$ and there exists $C > 0$ independent of $\a$ such that for all $z \in \C_+$ we have
\[
\nr{(\Wc_\a - z)\inv}_{\Lc(\HH)} \leq \frac C {\Re(z)} \left( 1 + \frac 1 {\abs {\a - N}} \right).
\]
\end{proposition}

\begin{proof}
Let $z \in \C_+$. Let $F = (f,g) \in \HH$. We prove that the equation 
\begin{equation} \label{eq:Wa-U-F}
(\Wc_\a - z) U = F
\end{equation}
has a unique solution $U = (u,v) \in \Dom(\Wc_\a)$. This will prove that $z \in \rho(\Wc_\a)$, and the explicit expression for $U$ will provide the estimates on $(\Wc_\a - z)\inv$. We identify $f \in \tilde H^1(\G)$ with any representative in $\dot H^1(\G)$. We can check all along the proof that if we add a constant to $f$, then it only changes $u$ by a constant and hence does not change its equivalence class in $\tilde H^1(\G)$. We could directly fix that $f(0) = 0$, but then the independence with respect to an additive constant would be less explicit.
For clarity of the exposition, 
we divide the proof into several steps 
distinguished by the bullet mark.

\stepp Assume that $U = (u,v) \in \Dom(\Wc_\a)$ satisfies \eqref{eq:Wa-U-F}. Let $j \in \Nc$. Then $u_j \in H^2_\loc(\R_+^*)$ and in the sense of distributions on $\R_+^*$ we have 
\begin{equation} \label{eq:u-g-f}
u_j'' - z^2 u_j = g_j + zf_j.
\end{equation}
We identify $u_j$ with its continuous representative. Then there exist $A_j,B_j \in \C$ such that, for all $x > 0$,
\begin{equation} \label{expr:uj}
u_j(x) = A_j e^{-zx} + B_j e^{zx} - \frac 1 {z^2} (\rho_z* g_j)(x) - \frac 1 z (\rho_z * f_j)(x),
\end{equation}
with $\rho_z$ defined by \eqref{def:rho-z}. 
Since $u_j$, $(\rho_z*g_j)$ and $(\rho_z*f_j)$ belong to $\dot H^1(\R_+^*)$ we necessarily have $B_j = 0$. If we set 
\begin{equation} \label{eq:tilde-Aj}
\tilde A_j = A_j + \frac {f(0)}{2z},
\end{equation}
then we have by Lemma \ref{lem:conv-dotH1}
\begin{equation} \label{expr:u'}
u_j'(x) =  - z \tilde A_j e^{- z x}  - \frac 1 {z} \big(\rho_z* (f_j' + g_j) \big) (x) + \int_0^x e^{-z(x-s)} g_j(s) \, ds
\end{equation}
and
\begin{equation} \label{expr:vj}
\begin{aligned}
v_j(x) 
& = zu_j(x) + f_j(x) \\
& = z \tilde A_je^{-z x} - \frac 1 {z} \big(\rho_z* (f_j' + g_j) \big) (x) + \int_0^x e^{-z(x-s)} f'_j(s) \, ds. 
\end{aligned}
\end{equation}
By continuity of $u$ at 0 we have for all $j \in \Nc$
\begin{equation} \label{expr:Aj}
\tilde A_j = u(0) + \frac {f(0)}{z} + \frac 1 {z^2} \big(\rho_z * (f_j'+g_j) \big)(0),
\end{equation}
so the Robin condition in \eqref{Dom-Wc} reads
\begin{align*}
\sum_{j=1}^N \left( - z u(0) - f(0) - \frac 2 {z} \big(\rho_z*(f_j'+g_j) \big)(0) \right) + \a z u(0) + \a f(0) = 0.
\end{align*}
This gives
\begin{equation} \label{expr:u0}
u(0) + \frac {f(0)} {z} =  \frac 2 {z^2(\a-N)} \sum_{\ell = 1}^N \big(\rho_z*(f_\ell' + g_\ell)\big)(0).
\end{equation}
Thus \eqref{expr:Aj} gives an explicit expression for $\tilde A_j$ and, 
by \eqref{expr:u'}--\eqref{expr:vj}, $U$ is uniquely determined by $F$. This proves the injectivity of $(\Wc_\a-z) : \Dom(\Wc_a - z) \to \HH$.   

\stepp 
Conversely, let $U = (u,v)$ be defined by \eqref{expr:uj}
and \eqref{expr:vj}, with $B_j = 0$ and $A_j$ given by \eqref{eq:tilde-Aj}, \eqref{expr:Aj} and \eqref{expr:u0}. Let $j \in \Nc$. Then $u_j$ is a solution of \eqref{eq:u-g-f}. By Lemma \ref{lem:conv-dotH1} we have $v_j \in H^1(\R_+^*)$ and $u_j' \in L^2(\R_+^*)$. Moreover, by \eqref{eq:u-g-f},
\[
u_j'' = zv_j + g_j \in L^2(\R_+^*).
\]
By construction, $u$ is continuous at 0, then so is $v$, and the Robin condition in \eqref{Dom-Wc} holds. This proves that $U \in \Dom(\Wc_\a)$ and then that $z \in \rho(\Wc_\a)$.

\stepp It remains to prove that 
\[
\nr{U}_\HH \lesssim \frac {\nr{F}_\HH}{\Re(z)} \left( 1 + \frac 1 {\abs {\a-N}} \right),
\]
where the symbol $\lesssim$ means that we have inequality up to a multiplicative constant which does not depend on $F$, $\a$ or $z$. For this we apply Lemma \ref{lem:conv-dotH1}. By \eqref{expr:u0} we have 
\[
\abs{u(0) + \frac {f(0)} {z}} \lesssim \frac {\nr{F}_\HH} {\abs z\abs {\a-N} \sqrt{\Re(z)}}.
\]
Then \eqref{expr:Aj} gives for all $j \in \Nc$ 
\[
|\tilde A_j| \lesssim \frac {\nr{F}_\HH} {\abs z \sqrt{\Re(z)}} \left( 1 + \frac 1 {\abs {\a-N}}\right).
\]
Finally, with \eqref{expr:u'} and \eqref{expr:vj},
\[
\nr{U}_\HH \lesssim \sum_{j=1}^N \big( \|u_j'\|_{L^2(\R_+^*)} + \|v_j\|_{L^2(\R_+^*)} \big) \lesssim \frac {\nr{F}_\HH} {\Re (z)} \left( 1 + \frac 1 {\abs {\a-N}}\right),
\]
and the proof is complete.
\end{proof}

        \detail 
        {
        $w = u'+zu$, $w'- zw  = \f$,
        \[
        w(t) = C_1 e^{zt} + \int_0^t e^{z(t-s)} \f(s) \, ds
        \]
        \begin{align*}
        u(x)
        & = C_2 e^{-zx} + \int_0^x e^{z(t-x)} w(t) \, dt\\
        & = C_2 e^{-zx} + \frac {C_1}{2z} (e^{zx}-e^{-zx}) + \int_{t=0}^x \int_{s=0}^t e^{z(2t-x-s)} \f(s)\, ds \, dt\\
        & = A e^{-zx} + Be^{zx} + \int_{s=0}^x \f(s) \int_{t=s}^x e^{z(2t-x-s)} \, dt\, ds\\
        & = \left( A - \frac 1 {2z} \int_{0}^x \f(s)  e^{zs} \, ds \right) e^{-zx} + \left( B + \frac 1 {2z} \int_{0}^x \f(s)  e^{-zs}  \, ds \right) e^{zx}\\
        \end{align*}
        Case $\Re(z) > 0$. We have to choose $B = -\frac 1 {2z} \int_{0}^\infty \f(s)  e^{-zs}  \, ds$
        \[
        u(x) = \left( A - \frac 1 {2z} \int_{0}^x \f(s)  e^{zs} \, ds \right) e^{-zx} - \frac {e^{zx}} {2z} \int_{x}^\infty \f(s)  e^{-zs}  \, ds  
        \]
        \[
        u'(x) = \left( -zA + \frac 1 {2} \int_{0}^x \f(s)  e^{zs} \, ds \right) e^{-zx} -  \frac {e^{zx}} {2} \int_{x}^\infty \f(s)  e^{-zs}  \, ds
        \]
        \[
        u_j(0) =  A_j  - \frac {1} {2z} \int_{0}^\infty \f_j(s)  e^{-zs}  \, ds 
        \]
        \[
        A_j = u(0) + \frac {1} {2z} \int_{0}^\infty \f_j(s)  e^{-zs}  \, ds.
        \]

        \[
        u_j'(0) = -zA_j  -  \frac {1} {2} \int_{0}^\infty \f_j(s)  e^{-zs}  \, ds
        \]
        }

\section{Damped wave equation}\label{Sec.evolution}

In this section we prove Theorem~\ref{th:damped-wave} 
about the time-dependant problem \eqref{eq:wave}--\eqref{eq:wave-CI}.

\begin{proof}[Proof of Theorem \ref{th:damped-wave}]
\stepp Assume that $u$ is a solution of \eqref{eq:wave}--\eqref{eq:wave-CI} on $[0,\tau[$ for some $\tau > 0$. For $t \in [0,\tau[$ and $j \in \Nc$ we identify $u_j(t)$ with its representative of class $C^1$ on $\R_+$. For $j \in \Nc$ there exist $\f_j \in C^1(-\tau,+\infty) \cap \dot H^1(-\tau,+\infty)\cap \dot H^2(-\tau,+\infty)$ and $\psi_j \in C^1(\R_+) \cap \dot H^1(\R_+) \cap \dot H^2(\R_+)$ such that $\f(0) = \psi(0) = 0$ and for all $t \in [0,\tau[$ and $x > 0$ we have
\begin{equation} \label{expr:uj-phipsi}
u_j(t,x) = u(0,0) + \f_j(x-t)  + \psi_j(x+t).
\end{equation}
The initial condition gives for $x > 0$
\[
\begin{cases}
\f_j'(x) + \psi_j'(x) = \partial_x u(0,x) = f_j'(x),\\
- \f_j'(x) + \psi_j'(x) = \partial_t u(0,x) = g_j(x),
\end{cases}
\]
so
\begin{equation} \label{expr:phi'-psi'}
\forall x > 0, \quad \psi_j'(x) = \frac {f_j'(x) + g_j(x)}{2}, \quad \f_j'(x) =  \frac {f_j'(x) - g_j(x)}{2}.
\end{equation}
On the other hand, the continuity at the central vertex gives for $j,k \in \Nc$ and $t > 0$, after differentiation,
\begin{equation} \label{eq:phipsi-continuite}
\forall j,k \in \Nc, \forall t \in [0,\tau[, \quad \quad \f_j'(-t) - \f_k'(-t) = \psi_j'(t) - \psi_k'(t).
\end{equation}
Finally, the damping condition yields
\begin{equation} \label{eq:phipsi-Robin}
\forall t \in [0,\tau[, \quad \left( \frac \a N - 1 \right) \sum_{j=1}^N  \f_j'(-t) = \left(  \frac \a N + 1 \right) \sum_{j=1}^N  \psi_j'(t).
\end{equation}
In particular, we see that if $\a = N$ then we necessarily have $\sum_{j=1}^N \psi_j'(t) = 0$ for all $t \in [0,\tau[$, so $\tau \leq t_0$. On the other hand, if $\a \neq N$ then $u$ is uniquely determined.

\stepp Conversely, we assume that $\a \neq N$ and prove that 
\eqref{eq:wave}--\eqref{eq:wave-CI} indeed has a solution on $\R_+$. For $j \in \Nc$ and $x \geq 0$ we set
\begin{equation} \label{expr:phi-psi}
\f_j(x) = \frac 12 \int_0^x \big( f_j'(s) - g_j(s) \big) \, ds, \quad \psi_j(x) = \frac 12 \int_0^x \big( f_j'(s) + g_j(s) \big) \, ds.
\end{equation}
and then $\f(x) = (\f_j(x))_{j \in \Nc}$ and $\psi(x) = (\psi_j(x))_{j \in \Nc}$, seen as column vectors. Then for $s \in \R_+^*$ we set 
\[
\f(-s) = - M_-\inv M_+ \psi(s),
\]
where 
\begin{equation}
M_\pm =
\begin{pmatrix}
1 & - 1 \\
&\ddots & \ddots \\
&&  1 & - 1 \\
\frac \a N \pm 1 &\dots & \frac \a N \pm 1 & \frac \a N \pm 1 
\end{pmatrix}.
\end{equation}
This is possible since $\det(M_-) = \a - N \neq 0$. We have $\f \in C^1(\R_-^*) \cap \dot H^1 (\R_-^*) \cap \dot H^2(\R_-^*)$. For $s > 0$ we have 
\[
\f(-s) = -M_-\inv M_+ \psi(s) \limt s 0 -M_-\inv M_+ \psi(0) = 0 = \f(0),
\]
and since $(f,g) \in \Dom(W_\a)$ 
\[
\f'(-s) = M_-\inv M_+ \psi'(s) \limt s 0 M_-\inv M_+ \psi'(0) = \f'(0^+).
\]
Thus $\f \in C^1(\R) \cap \dot H^1 (\R) \cap \dot H^2(\R)$.

Then for $t \geq 0$ and $x > 0$ we define $u$ by \eqref{expr:uj-phipsi} (the choice of $u(0,0)$ is not important since $u$ is defined up to an additive constant). We have $u \in C^1(\R_+, H^1(\G)) \cap C^0(\R_+, \dot H^2(\G^*))$. On the other hand in the sense of distributions we have 
\[
\partial_{tt} u = \partial_{xx} u.
\]
In particular, $\partial_{tt}u \in C^0(\R_+,L^2(\G))$, so $u \in C^2(\R_+,L^2(\G))$. All this proves that $u$ is a solution of \eqref{eq:wave}--\eqref{eq:wave-CI} on $\R_+$. Moreover for $t \geq 0$ we have 
\begin{align*}
E(u;t)
& = \nr{\partial_t u(t)}_{L^2(\G)}^2 + \nr{\partial_x u(t)}_{L^2(\G)}^2 \\
& = 2 \int_0^{+\infty} \big( \abs{\f'(x-t)}_{\C^N}^2 + \abs{\psi'(x+t)}_{\C^N}^2 \big) \, dx\\
& = 2 \int_{-t}^{+\infty} \abs{\f'(s)}_{\C^N}^2 \, ds + 2 \int_t^{+\infty} \abs{\psi'(s)}_{\C^N}^2\, ds.
\end{align*}
We have
\begin{align*}
\int_{-t}^t \abs{\f'(s)}_{\C^N}^2 \, ds 
& \leq \int_0^t \big( \abs{\f'(s)}_{\C^N}^2 + \nr{M_-\inv M_+}^2 \abs{\psi'(s)}_{\C^N}^2 \big) \, ds\\
& \lesssim \left( 1 + \frac 1 {\abs {\a - N}^2} \right) \int_0^t  \big( \abs{\f'(s)}_{\C^N}^2 + \abs{\psi'(s)}_{\C^N}^2 \big) \, ds\\
& \lesssim \left( 1 + \frac 1 {\abs {\a - N}^2} \right) \int_0^{t} \big( \abs{f'(s)}_{\C^N}^2 + \abs{g(s)}_{\C^N}^2 \big) \, ds.
\end{align*}
Since
\begin{align*}
2 \int_t^{+\infty} \big( \abs {\f'(s)}_{\C^N}^2 + \abs{\psi'(s)}_{\C^N}^2 \big) \, ds = \int_t^{+\infty} \big( \abs{f'(s)}_{\C^N}^2 + \abs{g'(s)}_{\C^N}^2 \big) \, ds,
\end{align*}
the desired inequality \eqref{eq:energy} follows.

\stepp Now assume that $\a = N$ and that $t_0 > 0$. Let $\tau \in ]0,t_0]$. On $\R_+$ we define $\f_j$ and $\psi_j$ by \eqref{expr:phi-psi}. Let $\th \in C^\infty([0,\tau[,\R)$ such that $\th(0) = 0$ and $\th'(0) = \f_1'(0)$. For $j \in \Nc$ and $s \in ]0,t_0[$ we set 
\[
\f_j(-s) = \th(-s) + \psi_j(s) - \psi_1(s).
\]
In particular $\f \in C^1(]-\tau,+\infty[;\C^N)$ and $\f \in \dot H^1(]-\tau_1,+\infty[;\C^N) \cap \dot H^2(]-\tau_1,+\infty[;\C^N)$ for all $\tau_1 \in ]0,\tau[$. Finally we define $u$ by \eqref{expr:uj-phipsi}. Then $u$ is a solution of \eqref{eq:wave}--\eqref{eq:wave-CI} on $[0,\tau[$. Moreover for $t \in [0,\tau[$ we have 
\begin{align*}
E(u,t) 
\geq  \int_0^{+\infty} \big( \abs{\f_1'(x-t)}^2 + \abs{\psi_1'(x+t)}^2 \big) \, dx 
\geq \int_{-t}^0 \abs{\f_1'(s)}^2 \, ds \geq \int_0^t \abs{\th'(s)} \,ds.
\end{align*}
In particular we can choose $\th$ in such a way that $E(u;t)$ goes to $+\infty$ as $t$ goes to $\tau$.
\end{proof}

\section{Relationship with relativistic quantum mechanics}\label{Sec.Dirac}
%
If the space variable~$x$ is restricted to a bounded interval
and~$\alpha$ is real,  
the classical interpretation of the wave equation~\eqref{wave}
is that of a vibrating string, 
subject to a highly localised damping~\cite{Bamberger-Rauch-Taylor_1982} .
Strictly speaking, the genuine damping (or friction) corresponds to 
negative~$\alpha$, while positive~$\alpha$ models a supply
of energy into the system. 
However, we use the word ``damping'' even in the more general 
situation whenever~$\alpha$ has a non-zero real part,
so that the physical system is not conservative. 
  
The goal of this last section is to provide 
a physical motivation in the unconventional setting 
of unbounded geometries and/or non-real~$\alpha$.
To this purpose, we recall a more or less well known 
relationship between the damped wave equation 
and the Dirac equation in relativistic quantum mechanics
(see, e.g., \cite{Gesztesy-Holden_2011,Gesztesy_2012,Cuenin-Siegl_2018}).
(Here the word relativistic stands for the original usage of the equation,
recent years have brought motivations to consider the same 
mathematical problem for non-relativistic systems like graphene.)

For simplicity, let us restrict to the simplest graph $N=2$,
which can be identified with the real axis~$\R$.
Instead of substituting $(u,v)=(\psi,\partial_t\psi)$ 
as above \eqref{eq:wave-Wc},
let us write $(\xi,\eta)=(\partial_t\psi,\partial_x\psi$).
Then the wave equation~\eqref{wave} is formally transferred 
to the first-order system 
\begin{equation}\label{wave.Dirac}
  \partial_t
  \begin{pmatrix}
    \xi \\ \eta
  \end{pmatrix}
  =
  i \mathcal{D}_\alpha 
  \begin{pmatrix}
    \xi \\ \eta
  \end{pmatrix}
  \qquad\mbox{with} \qquad
  \mathcal{D}_\alpha =
  \begin{pmatrix}
    -i\alpha\delta & -i\partial_x \\
    -i\partial_x & 0
  \end{pmatrix}
  .
\end{equation}
Here $\mathcal{D}_\alpha$ is a Dirac-type operator 
considered in the Hilbert space $L^2(\R) \times L^2(\R)$. 
More specifically, $\mathcal{D}_0$ with domain $H^1(\R) \times H^1(\R)$
is the (self-adjoint) Dirac Hamiltonian modelling the propagation of
relativistic (quasi-)particles in quantum mechanics. 
The perturbation 
$
\begin{psmallmatrix}
    -i\alpha\delta & 0 \\
    0 & 0
\end{psmallmatrix}
$
(properly introduced via the jump condition 
$\eta(0^+)-\eta(0^-)=-\alpha \xi(0)$
together with the continuity $\xi(0^+) = \xi(0^-) = \xi(0)$)
represents neither a purely electric nor scalar potential,
but it is self-adjoint 
(and therefore quantum-mechanically relevant)
whenever~$\alpha$ is purely imaginary. 
Moreover, the real part of~$\alpha$ is potentially eligible
in quasi-Hermitian quantum mechanics~\cite{KSTV}. 

It is interesting to compare the present model~\eqref{wave.Dirac}
with the $\delta$-shell interaction Hamiltonian
\begin{equation}\label{wave.Dirac.bis}
  \tilde{\mathcal{D}}_\tau =
  \begin{pmatrix}
    \tau\delta & -i\partial_x \\
    -i\partial_x & \tau\delta
  \end{pmatrix}
\end{equation}
intensively studied for real~$\tau$
during the last decade  
(see \cite{Arrizibalaga-Mas-Vega_2014,Ourmieres-Vega_2018,
Holzmann-Ourmires-Bonafos-Pankrashkin_2018,
Behrndt-Holzmann_2020}
for counterparts of~\eqref{wave.Dirac.bis} in~$\R^3$ and its variants).
Here $\tilde{\mathcal{D}}_\tau$ is properly introduced 
via the transmission condition
$$
  \begin{pmatrix}
    \tau/2 & -i \\
    -i & \tau/2
  \end{pmatrix}
  \begin{pmatrix}
    \xi(0^+) \\ \eta(0^+)
  \end{pmatrix}
  +
  \begin{pmatrix}
    \tau/2 & i \\
    i & \tau/2
  \end{pmatrix}
  \begin{pmatrix}
    \xi(0^-) \\ \eta(0^-)
  \end{pmatrix}
  = 
  \begin{pmatrix}
    0 \\ 0
  \end{pmatrix}
$$
that the functions from the operator domain 
of~$\tilde{\mathcal{D}}_\tau$ must satisfy.
This definition makes sense even for complex~$\tau$.
If~$\tau$ is real, the operator $\tilde{\mathcal{D}}_\tau$
is self-adjoint and its (purely continuous) spectrum coincides
with the real axis~$\R$. In fact, the same spectral picture 
extends to be true for any $\tau \in \C\setminus\{\pm 2i\}$.
However, if $\tau = 2i$ (respectively, $\tau = -2i$), 
then the whole upper (respectively, lower) complex plane
belongs to the spectrum of~$\tilde{\mathcal{D}}_\tau$.
So we again see the presence of the not-any-more-magical number~$2$.

\vspace{-3ex}
\subsection*{Acknowledgments} 
The first author (D.K.) was supported by the EXPRO grant No.~20-17749X 
of the Czech Science Foundation. The second author (J.R.) was supported by the ANR LabEx CIMI (grant ANR-11-LABX-0040) within the French State Programme "Investissements d'Avenir".

\vspace{-3ex}
%
\providecommand{\bysame}{\leavevmode\hbox to3em{\hrulefill}\thinspace}
\providecommand{\MR}{\relax\ifhmode\unskip\space\fi MR }
\providecommand{\MRhref}[2]{%
  \href{http://www.ams.org/mathscinet-getitem?mr=#1}{#2}
}
\providecommand{\href}[2]{#2}

\end{document}